\documentclass[12pt]{article}

\usepackage{amsmath}
\usepackage{amsthm}
\usepackage{amsfonts}
\usepackage{mathrsfs}
\usepackage{stmaryrd}
\usepackage{setspace}
\usepackage{fullpage}
\usepackage{amssymb}
\usepackage{breqn}
\usepackage{enumitem}
\usepackage{bbold}
\usepackage{authblk}
\usepackage{comment}
\usepackage{hyperref}
\usepackage{pgf,tikz}
\usepackage{graphicx}
\usepackage{subcaption}

\newtheorem{thm}{Theorem}[section]

\newtheorem{proposition}[thm]{Proposition}
\newtheorem{prop}[thm]{Proposition}

\newtheorem{corollary}[thm]{Corollary}

\newtheorem{clm}[thm]{Claim}

\newcommand\ex{\ensuremath{\mathrm{ex}}}

\newcommand\cG{{\mathcal G}}
\newcommand\cH{{\mathcal H}}

\newcommand\cL{{\mathcal L}}

\newcommand\cN{{\mathcal N}}

\newcommand\cS{{\mathcal S}}

\renewcommand{\S}{\mathcal S}

\newtheorem*{thm*}{Theorem}
\newtheorem*{prop*}{Proposition}
\newcommand{\ignore}[1]{}

\title{On Tur\'{a}n problems for suspension hypergraphs}
\author{
Xin Cheng\thanks{\small School of Mathematics and Statistics, Northwestern Polytechnical University and Xi'an-Budapest Joint Research Center for Combinatorics, Xi'an 710129, Shaanxi, P.R. China. Email:
\small \texttt{xincheng@mail.nwpu.edu.cn}.}\,, \hspace{0.2em}
D\'{a}niel Gerbner\thanks{\small Alfr\'ed R\'enyi Institute of Mathematics. Email:
\small \texttt{gerbner.daniel@renyi.hu}.}\,, \hspace{0.2em} 
Hilal Hama Karim\thanks{\small Department of Computer Science and Information Theory, Faculty of Electrical Engineering and Informatics, Budapest University of Technology and Economics, Műegyetem rkp. 3., H-1111 Budapest, Hungary. E-mail: \texttt{hilal.hamakarim@edu.bme.hu}.}\,, \hspace{0.2em} 
Junpeng Zhou\thanks{\small Department of Mathematics, Shanghai University, Shanghai 200444, P.R. China. Email:
\small \texttt{junpengzhou@shu.edu.cn}.} \thanks{\small Newtouch Center for Mathematics of Shanghai University, Shanghai 200444, P.R. China.}}

\date{}

\begin{document}

\maketitle

\begin{abstract}
For a given graph $F$, the $r$-uniform suspension of $F$ is the $r$-uniform hypergraph obtained from $F$ by taking $r-2$ new vertices and adding them to every edge. In this paper, we consider Tur\'{a}n problems on suspension hypergraphs, and we obtain several general and exact results. 
\end{abstract}



\section{Introduction}
A \textit{hypergraph $\cH=(V(\cH),E(\cH))$} consists of a vertex set $V(\cH)$ and a set $E(\cH)$ of hyperedges, where each hyperedge in $E(\cH)$ is a nonempty subset of $V(\cH)$. If $|e|=r$ for every $e\in E(\cH)$, then $\cH$ is called an \textit{$r$-uniform hypergraph} ($r$-graph for short). The \textit{degree} $d_\cH(v)$ of a vertex $v$ is the number of hyperedges containing $v$ in $\cH$.

Let $\mathcal{F}$ be an $r$-graph.
A hypergraph $\cH$ is called \textit{$\mathcal{F}$-free} if $\cH$ does not contain $\mathcal{F}$ as a subhypergraph. 
The \textit{Tur\'{a}n number} ${\rm{ex}}_r(n,\mathcal{F})$ of $\mathcal{F}$ is the maximum number of hyperedges in an $\mathcal{F}$-free $r$-graph on $n$ vertices. When $r=2$, we write $\mathrm{ex}(n,\mathcal{F})$ instead of $\mathrm{ex}_2(n,\mathcal{F})$.

Hypergraph Tur\'{a}n problems are among central topics in extremal combinatorics \cite{FS} and are notably more challenging than their graph counterparts. Some natural questions are surprisingly difficult. For example, the Tur\'{a}n number of $K_4^3$, the 3-uniform 4-vertex complete hypergraph, remains unknown. In this paper we study Turán problems for a specific class of hypergraphs, called suspensions.

For a given graph $F$, the \textit{$r$-uniform suspension} (or briefly $r$-suspension) $\S^r F$ of $F$ is the $r$-graph obtained from $F$ by taking $r-2$ new vertices and adding them to every edge. Sidorenko \cite{Si} first introduced this notion for trees, referring to it as \textit{enlargement}. In Keevash's survey \cite{Ke}, it is termed \textit{extended $k$-trees}, but when applied to complete graphs, it is referred to as suspension or, alternatively, as \textit{cone}. Similarly, S\'{o}s, Erd\H{o}s and Brown \cite{SEB} also use the term cone and specifically call $\S^3 C_{\ell}$ a \textit{wheel}. Alon \cite{alon} uses the term \textit{augmentation}.  Bollob\'{a}s, Leader and Malvenuto \cite{BLM} studied $\S^rK_k$, naming it \textit{daisy}. In \cite{GPa}, Gerbner and Patk\'os uniformly refer to them as suspension and collect the results on Tur\'{a}n problems on suspensions in Subsection 5.2.5. 

There are several results that ``accidentally'' deal with suspensions. For example, one of the most studied hypergraph Tur\'an problems concerns $K_4^{3-}$, which is obtained from $K_4^3$ by deleting an edge. The best known bounds are $2\binom{n}{3}/7\le \ex_3(n,K_4^{3-})\le 0.2871\binom{n}{3}$, due to Frankl and F\"uredi \cite{FF1} and Baber and Talbot \cite{BT}, respectively. Observe that $K_4^{3-}=\S^3K_3$. More generally, the unique $r$-graph consisting of 3 hyperedges on $r+1$ vertices is $\S^rK_3$. The current best bound on their Tur\'an number for general $r$ is due to Sidorenko \cite{Si2}. In the case $r=4$, Gunderson and Semeraro \cite{GS} showed $\ex_4(n,\S^4K_3)\le n\binom{n}{3}/16$, and that this bound is sharp for infinitely many $n$, and asymptotically sharp in general. 

Bollob\'{a}s, Leader and Malvenuto \cite{BLM} studied $\S^rK_k$, focusing on the case $k=4$. They conjectured that $\ex_r(n,\S^rK_k)/\binom{n}{r}$ goes to 0 as $r$ increases, but this was disproved by Ellis, Ivan and Leader \cite{EIL}. Other papers deal with the Tur\'an number of the suspension of complete bipartite graphs \cite{Mu} or even cycles \cite{Muk}.

\smallskip

Let us introduce some definitions that are used throughout the paper. Given an $r$-graph $\cH$ and a set $S$ of $k$-elements ($k\leq r-2$), the \textit{link hypergraph} $\cL(\cH,S)$ is the $(r-k)$-graph obtained by removing $S$ from the hyperedges of $\cH$ that contain $S$. If $S=\{v\}$, then we write $\cL_v$ instead of $\cL(\cH,\{v\})$. 
In the case $r-k=2$, we call $\cL(\cH,S)$ the \textit{link graph} of $S$. Then an $r$-graph is $\cS^rF$-free if and only if the link graph of every $(r-2)$-set is $F$-free.

An \textit{$(n,r,k,\lambda)$-design} is an $n$-element $r$-graph with the property that each $k$-element set of vertices is contained in exactly $\lambda$ hyperedges. It is easy to see and well-known that some divisibility conditions are implied. The Existence Theorem due to Keevash \cite{Ke2} states that if these necessary divisibility conditions are satisfied and $n$ is sufficiently large, then there exists an $(n,r,k,\lambda)$-design.

Most of the papers dealing with suspensions use some specialized version of the following simple observation. 

\begin{proposition}\label{simpli}
    Let $k<r$. For any graph $F$ and $n$, we have $\ex_r(n,\cS^rF)\le \binom{n}{r-k}\ex_k(n-r+k,\cS^kF)/\binom{r}{k}$. In particular, $\ex_r(n,\cS^rF)\le \binom{n}{r-2}\ex(n-r+2,F)/\binom{r}{2}=(1+o(1))2\binom{n}{r}\ex(n,F)/n^2$. 
\end{proposition}

Now let us consider trees with $t$ edges. In the graph case, the Erd\H os-S\'os conjecture \cite{Er} states that $\ex(n,T)\le (t-1)n/2$. The conjecture is known to hold for several classes of trees, e.g., for paths by the Erd\H os-Gallai theorem \cite{EGa} and for spiders \cite{FHL}.  
A \textit{hypertree} is a hypergraph obtained by adding hyperedges one by one the following way. Each hyperedge other than the first has one new vertex and the rest of its vertices are contained in a hyperedge added earlier. Note that it is also called the \textit{tight $r$-tree}. Kalai conjectured (see \cite{FF2}) that if $T^r$ is an $r$-uniform hypertree with $t$ hyperedges, then $\ex_r(n,T^r)\le (t-1)\binom{n}{r-1}/r$. Kalai's conjecture is also known to hold for various classes of hypertrees, see e.g. \cite{FJKMV1}.

It is easy to see that graph trees are hypertrees. The suspension of a tree is also a hypertree by the same ordering of the edges. Applying Proposition \ref{simpli}, we obtain that if a tree satisfies the Erd\H os-S\'os conjecture, then $\cS^rT$ satisfies Kalai's conjecture. For other trees with $t$ edges, the trivial bound is $\ex(n,T)\le (t-1)n$, which is improved to $t(t-1)n/(t+1)$ in \cite{Ste}. For hypertrees $T^r$, the trivial bound is $\ex_r(n,T^r)\le (t-1)\binom{n}{r-1}$. Combining Proposition \ref{simpli} with the bounds in the graph case, we obtain that for any tree $T$ with $t$ edges, $\ex_r(n,\cS^rT)\le 2t(t-1)\binom{n}{r-1}/(r(t+1))$.

Further, let us mention that the upper bound $\ex_r(n,\cS^rT)\le (t-1)\binom{n}{r-1}/r$ is sharp for the tree $T$ that satisfies the Erd\H os-S\'os conjecture and infinitely many $n$, as shown by the following construction. We take an $(n,r+t-2,r-1,1)$-design, i.e., an $n$-vertex $(r+t-2)$-graph where each $(r-1)$-set is contained in exactly one hyperedge. Such a hypergraph exists by the Existence Theorem of Keevash \cite{Ke2} if $n$ is sufficiently large and some necessary divisibility conditions are satisfied. Then we take the $r$-subsets of each hyperedge as hyperedges of a new $r$-graph. The resulting hypergraph has the desired number of hyperedges and for each $(r-2)$-set, its link graph consists of vertex-disjoint copies of $K_t$, thus the link graph is $T$-free.

\smallskip

Observe that Proposition \ref{simpli} implies that $\ex_r(n,\cS^rF)=O(n^{r-2}\ex(n,F))$. For each graph $F$ we have mentioned so far, this bound is sharp. In fact, the only example we know where this is not sharp is the matching $M_t$ with $t$ edges. Observe that $\cS^rM_t$ consists of $t$ hyperedges having the same $r-2$ vertices as pairwise intersections (such a hypergraph is called a \textit{sunflower}). Such intersection problems are widely studied, see e.g. \cite{GPa}. Frankl and F\"uredi \cite{FF} showed that $\ex_r(n,\cS^rM_2)=(1+o(1))6\binom{n}{r-2}/(r(r-1))$. Chung and Frankl \cite{CF} determined $\ex_3(n,\cS^3M_t)$. In particular, $\ex_3(n,\cS^3M_t)=O(n)$, which together with Proposition \ref{simpli} implies that $\ex_r(n,\cS^rM_t)=O(n^{r-2})$ for $r\ge 3$. Note that if we take each $r$-set contained in the hyperedges of an $(n,r+2t-3,r-2,1)$-design, then we obtain an $\cS^rM_t$-free hypergraph. Frankl and F\"uredi \cite{FF2} conjecture that this lower bound is asymptotically sharp.
Recently, Brada\v{c}, Buci\'{c} and Sudakov \cite{BBS} determined the correct order of magnitude of the Tur\'{a}n number of sunflowers of arbitrary uniformity, in terms of both the size of the underlying graph and the size of the sunflower. In particular, they obtained $\ex_r(n,\cS^rM_t)=\Theta(n^{r-2})$ for $r\geq4$.

Next, we present an upper bound for the Tur\'{a}n number of suspensions of graphs with chromatic number $k$. 

\begin{thm}\label{asy}
    For any graph $F$ with $\chi(F)=k$, we have $\ex_r(n,\cS^rF)\le \ex_r(n,\cS^rK_k)+o(n^r)$.
\end{thm}

In the graph case, Theorem \ref{asy} is equivalent to the Erd\H os-Stone-Simonovits theorem \cite{ES1966,ES1946}, which determines the asymptotics of $\ex(n,F)$ for every non-bipartite graph $F$. For larger uniformity, not many asymptotic results are implied, since we do not even know the asymptotics of $\ex_r(n,\cS^rK_k)$, with the exception of $r=4$, $k=3$. As we have mentioned, Gunderson and Semeraro \cite{GS} showed $\ex_4(n,\cS^4K_3)\le n\binom{n}{3}/16$, and that this bound is sharp for infinitely many $n$, and asymptotically sharp in general. Moreover, Proposition 23 in \cite{GS} shows that their construction is $\cS^4F$-free for every non-bipartite graph $F$. This together with Theorem \ref{asy} implies the following result.

\begin{corollary}
    For any graph $F$ with chromatic number 3, we have $\ex_4(n,\S^4F)= (1+o(1))n\binom{n}{3}/16$.
\end{corollary}

We improve the above result for odd cycles to a sharp bound 
for infinitely many $n$.

\begin{thm}\label{shar}
  
   For every $k\geq 1$, we have $\ex_4(n,\cS^4C_{2k+1})\le n\binom{n}{3}/16$ for $n$ sufficiently large. This bound is sharp for infinitely many $n$.
\end{thm}

Let us consider the smallest forest that is not a matching, i.e., the disjoint union of $P_3$ and $K_2$, which we denote by $P_3\cup K_2$. The $3$-graph $\cS^3(P_3\cup K_2)$ has 3 hyperedges on 6 vertices, thus it is one of the forbidden hypergraphs in the celebrated 6-3 theorem of Ruzsa and Szemer\'edi \cite{RS}. Surprisingly, we did not find any result on the Tur\'an number of $\cS^3(P_3\cup K_2)$ alone. Here, we present a sharp result.

\begin{thm}\label{fores} 
   For every $n>33$, $\ex_3(n,\cS^3(P_3\cup K_2)\le \binom{\lfloor (3n-1)/4\rfloor}{2}/3+(3n-1)(n+1)/32$.
   
   Moreover, this bound is sharp apart from a constant additive term, and if $n=8k+5$, then $\ex_3(n,\cS^3(P_3\cup K_2))= \left\lfloor\binom{\lfloor (3n-1)/4\rfloor}{2}/3+(3n-1)(n+1)/32\right\rfloor$.
\end{thm}

The rest of this paper is organized as follows. In Section 2, we prove upper bounds for the Tur\'{a}n numbers of suspensions of a class of graphs, and provide the proofs of Theorems \ref{asy} and \ref{shar}. The proofs of Proposition \ref{simpli} and Theorem \ref{fores} are provided in Section 3. 

\section{Proofs of Theorems \ref{asy} and \ref{shar}}

In the proof of Theorem \ref{asy}, will apply the hypergraph removal lemma \cite{Gow,RSc}.
 
It states that if an $n$-vertex $r$-graph contains $o(n^{|V(\cH)|})$ copies of a hypergraph $\cH$, then all the copies of $\cH$ can be removed by removing $o(n^r)$ hyperedges. In order to use it, we have to show that $\cS^rF$-free $n$-vertex $r$-graphs contain $o(n^{r+k-2})$ copies of $\S^rK_k$. 

For two $r$-graphs $\cH$ and $\cH'$, let $\cN(\cH,\cH')$ denote the number of copies of $\cH$ in $\cH'$, and let $\ex_r(n,\cH,\cH')$ denote the maximum possible number of copies of $\cH$ in an $n$-vertex $\cH'$-free $r$-graph. When $r=2$, we write $\ex(n,\cH,\cH')$ instead of $\ex_2(n,\cH,\cH')$.
For two graphs $G$ and $G'$, Alon and Shikhelman \cite{AS} characterized when $\ex(n,G,G')=o(n^{|V(G)|})$. First, we extend this to hypergraphs. A \textit{blowup} of an $r$-graph $\cH$ is an $r$-graph obtained from $\cH$ by replacing each vertex $v\in V(\cH)$ by a new independent set $W(v)$ of size at least one, and each hyperedge $v_1v_2\dots v_r$ by a complete $r$-partite $r$-graph with parts $W(v_1),W(v_2),\dots, W(v_r)$. The \textit{$s$-blowup} of $\cH$ is the blowup that satisfies $|W(v)|=s$ for each vertex $v\in V(\cH)$. 

\begin{proposition}\label{alsh}
    Let $H$ and $H'$ be two $r$-graphs. Then $\ex_r(n,H,H')=\Omega(n^{|V(H)|})$ if and only if $H'$ is not a subhypergraph of a blowup of $H$. Otherwise, $\ex_r(n,H,H')=O(n^{|V(H)|-\varepsilon})$ for some $\varepsilon:=\varepsilon(H,H')>0$. 
\end{proposition}

The proof closely follows the argument for the case $r=2$ in \cite{AS}, we include it here for completeness. 

\begin{proof}[\bf Proof]
  If $H'$ is not a subhypergraph of a blowup of $H$, then the $\lfloor n/|V(H)|\rfloor$-blowup of $H$ is $H'$-free and contains $\Omega(n^{|V(H)|})$ copies of $H$.

    Assume now that $H'$ is a subhypergraph of a blowup of $H$. Then there is a smallest $s:=s(H,H')$ such that $H'$ is a subhypergraph of the $s$-blowup of $H$. Consider an $n$-vertex $H'$-free hypergraph $\cH$ and let $m$ denote the number of copies of $H$ in $\cH$. Now we consider a random partition of $V(\cH)$ to $t:=|V(H)|$ parts $V_1,\dots,V_t$. Let $u_1,\dots, u_t$ be the vertices of $H$, and we consider the copies of $H$ in $\cH$ where $u_i$ belongs to $V_i$ for every $i$. We call such copies of $H$ \textit{nice}. The expected number of nice copies is $m/t^t$, thus there is a partition such that at least $m/t^t$ copies of $H$ is nice. We fix that partition $V_1,\dots,V_t$.

    Now we construct a $t$-partite $t$-graph $\cH'$ with vertex set $V(\cH)$ where the parts are from the above partition. A $t$-set with vertices $v_i\in V_i$ for $i\le t$ forms a hyperedge of $\cH'$ if $\cH$ contains a copy of $H$ on these vertices, with $v_i$ playing the role of $u_i$. 
    In other words, a $t$-set forms a hyperedge of $\cH'$ if there is on these $t$ vertices one of the at least $m/t^t$ copies of $H$ that are nice with respect to the fixed partition.
    
    Therefore, $\cH'$ has at least $m/t^t$ hyperedges. Assume that $\cH'$ contains a complete $t$-partite subhypergraph of $\cH'$ with each part $U_i\subset V_i$ of order $s$. This means that for each hyperedge $u_{i_1}\dots u_{i_r}$ of $H$, and each vertex $v_{i_j} \in U_{i_j}$, the hyperedge $v_{i_1}\dots v_{i_r}$ is in $\cH$, otherwise extending these vertices with vertices form the other parts, we would not get a nice copy of $H$. Therefore, there is an $s$-blowup of $H$ in $\cH$ with parts $U_1,\dots, U_t$, a contradiction.
    
    We obtained that $\cH'$ does not contain a complete $t$-partite $t$-graph with parts of size $s$. By a result of Erd\H os \cite{erd}, there exists an $\varepsilon:=\varepsilon (s,t)>0$ such that $\cH'$ contains at most $n^{t-\varepsilon}$ hyperedges. Therefore, $\cN(H,\cH)=m\le t^tn^{t-\varepsilon}$, completing the proof.
\end{proof}

We are now ready to prove Theorem \ref{asy}. Recall that it states that for any graph $F$ with $\chi(F)=k$, we have $\ex_r(n,\cS^rF)\le \ex_r(n,\cS^rK_k)+o(n^r)$.

\begin{proof}[\bf Proof of Theorem \ref{asy}]
    First we show that for any $F$ with chromatic number $k$ and any $r$, we have that a blowup of $\S^rK_k$ contains $\S^rF$. Indeed, we can blow up the vertices of $K_k$ to obtain a graph $F'$ that contains $F$. Then the additional $r-2$ vertices of $\cS^rK_k$ with the edges of $F'$ form $\S^rF'$, which contains $\S^rF$.

    Using Proposition \ref{alsh}, we have that any $\S^rF$-free $r$-graph $\cG$ contains $o(n^{|V(\S^rK_k)|})$ copies of $\S^rK_k$. By the hypergraph removal lemma, we can remove those copies of $\cS^rK_k$ by removing $o(n^r)$ hyperedges. The resulting hypergraph $\cG'$ is $\S^rK_k$-free, and thus has at most $\ex_r(n,\S^rK_k)$ hyperedges. Therefore, $\cG$ has at most $|E(\cG')|+o(n^r)\le\ex_r(n,\cS^rK_k)+o(n^r)$ hyperedges, completing the proof.
\end{proof}

\smallskip

Let us now work toward the proof of Theorem \ref{shar}. We will prove the statement for a larger class of graphs.

Let $H(t)$ denote the following graph. We take two copies of $K_{t,t}$ and a triangle. We take three partite sets of the two copies of $K_{t,t}$ and take a vertex of the triangle for each of those parts. Then we join each vertex of those parts to the corresponding vertex in the triangle. In other words, two vertices of the triangle extend one of the copies of $K_{t,t}$ to a $K_{t+1,t+1}$, while the third vertex extends the other copy of $K_{t,t}$ to a $K_{t,t+1}$.  Let $Q(t)$ denote the following graph. We take independent sets $U_1,U_2,U_3,U_4$ each of order $t$, and a triangle with vertices $v_1,v_2,v_3$. For each $i\le 3$, we join $v_i$ and each vertex of $U_4$ to each vertex of $U_i$.

\begin{proposition}\label{triv}
For any $t$ there exists a constant $c:=c(t)<1/4$ such that if $n$ is large enough and an $n$-vertex graph $G$ contains a triangle but does not contain $H(t)$ nor $Q(t)$, then it contains a vertex of degree at most $cn$.
\end{proposition}

We remark that this statement cannot be extended to other graphs besides subgraphs of $H(t)$ and $Q(t)$, since $H(\lfloor (n-3)/4\rfloor)$ and $Q(\lfloor (n-3)/4\rfloor)$ have a minimum degree of at least $\lfloor (n-3)/4\rfloor$. We will use the K\H ov\'ari-S\'os-Tur\'an theorem \cite{KST}, which shows that $\ex(n,K_{t,t})\le cn^{2-1/t}$ for some constant $c$.

\begin{proof}[\bf Proof] 
Let $d$ denote the smallest degree in $G$. 
Let $H'(t)$ denote the graph we obtain from $H(t)$ by deleting the copy of $K_{t,t}$ that has only one partite set joined to some vertices of the triangle. First observe that if we find $H'(t)$ in $G$, then we are done. Indeed, let us assume that we found a copy of $H'(t)$ and let $v$ be the vertex of degree 2 there, i.e., the vertex of the triangle not joined to any vertex of the $K_{t,t}$. Let $A$ be the set of neighbors of $v$ that are not in the copy of $H'(t)$ and let $B$ be the set of vertices that are adjacent to some vertex of $A$ and are not in the copy of $H'(t)$. If $d=\Omega(n)$, then there are $\Omega(n)$ vertices in $A$, incident to $\Omega(n^2)$ edges. Then there is a copy of $K_{t,t}$ consisting of those edges in $A\cup B$, and one part of this $K_{t,t}$ is in $A$, thus this $K_{t,t}$ extends the copy of $H'(t)$ to a copy of $H(t)$, a contradiction.

Let $v_1v_2v_3$ be a triangle in $G$. Let $U_i$ denote the set of neighbors of $v_i$ that are not in this triangle. 
    
\begin{clm}
    \textbf{(i)} $|U_i\cap U_j|=o(n)$. 

    \textbf{(ii)} There are $o(n^2)$ edges inside $U_1\cup U_2\cup U_3$.
\end{clm}
    
\begin{proof}[\bf Proof of Claim]
Assume that $|U_i\cap U_j|=\Omega(n)$ and let $G_0$ denote the graph formed by the edges that are incident to $U_i\cap U_j$. Then $G_0$ has $\Omega(n^2)$ edges, thus contains $K_{t',t'}$ where $t'=|V(H(t))|$. Clearly, all the vertices in one part belong to $U_i\cap U_j$, thus the copy $K_{t',t'}$ together with $v_i$ and $v_j$ form a $K_{t',t'+2}$ with an additional edge inside the larger part. This contains $H'(t)$, a contradiction proving \textbf{(i)}.

To prove \textbf{(ii)}, we show that there is no $K_{t+1,t+1}$ inside $U_i$, and no $K_{t+1,t+1}$ in $G$ with one part in $U_i$ and other part in $U_j$. Indeed, the first one clearly contains a copy of $H'(t)$ with $v_i$, while the second one is a $K_{t+2,t+2}$ with $v_i$ and $v_j$, and the third vertex of the triangle extends it to a $H'(t+1)$. 
\end{proof} 

Let us return to the proof of the proposition and let $U_i'$ denote the set of vertices in $U_i$ that do not belong to any other $U_j$ and have $o(n)$ neighbors inside $U_1\cup U_2\cup U_3$. Then by the above claim, $|U_i'|=d-o(n)$. Let $U=V(G)\setminus (U_1\cup U_2\cup U_3\cup\{v_1,v_2,v_3\})$. Then each vertex of $U_i'$ has at least $d-o(n)$ neighbors in $U$, in particular $|U|\ge d-o(n)$. If $|U|\ge d+\Omega(n)$, then we are done, since in that case $n\ge |U_1'|+|U_2'|+|U_3'|+|U|\ge 4d+\Omega(n)$, thus $d\le n/4-\Omega(n)$. Hence we can assume that $|U|=d-o(n)$, which implies that each vertex of $U_i'$ is adjacent to all but $o(n)$ vertices of $U$.

Let us take $t$ vertices of each $U_i'$.
We claim that these $3t$ vertices have at least $t$ common neighbors in $U$. Indeed, each vertex of $U$ is a common neighbor of them except for the $o(n)$ vertices that are not adjacent to one of these $3t$ vertices. These $3t$ vertices with $t$ vertices from their common neighborhood in $U$ and with $v_1,v_2,v_3$ form a copy of $Q(t)$, a contradiction. 
\end{proof}

Let us remark that the above proposition implies a sharp bound on the Tur\'an number in the graph case.
Indeed, if we have an $n$-vertex $H(t)$-free and $Q(t)$-free graph with more than $\lfloor n^2/4\rfloor$ edges, then we start by repeatedly deleting the vertices of degree less than $cn$. This results in a graph on $n_0$ vertices with minimum degree at least $cn$. It is easy to see that this graph contains more than $\lfloor n_0^2/4\rfloor$ edges, thus contains a triangle, contradicting Proposition \ref{triv}.
This argument does not work for $r>2$, since we need to consider degrees of $(r-1)$-sets, and we cannot just delete $(r-1)$-sets. Instead, we obtain a similar result by a more complicated proof.

\begin{thm}\label{3chrom} 
Assume that $\ex(n,K_3,F)=o(n^2)$ and if an $F$-free $n$-vertex graph contains a triangle, then it has a vertex of degree less than $cn$ for some $c<1/r$. Then $\ex_r(n,\cS^rF)\le n\binom{n}{r-1}/r^2$ for $n$ sufficiently large.
\end{thm}

\begin{proof}[\bf Proof]
Let $\cH$ be an $n$-vertex $\cS^rF$-free $r$-graph with at least $n\binom{n}{r-1}/r^2$ hyperedges. We follow the proof of Gunderson and Semeraro \cite{GS}. It involves double counting the set $\{(A,B): A\in \cH, B\not\in \cH, B \text{ is an $r$-set}, |A\cap B|=r-1\}$. Let $t$ denote the number of copies of $\cS^rK_3$ in $\cH$.

A lower bound $|E(\cH')|(n-r)(r-1)$ in an $\cS^rK_3$-free hypergraph $\cH'$ is obtained in \cite{GS} by a simple argument: For each hyperedge $A\in\cH'$, there are $n-r$ ways to pick a vertex $b\not\in A$, and then at least $r-1$ ways to pick a vertex $a\in A$ such that $A\cup \{b\}\setminus \{a\}$ is not in $\cH'$. This holds because in the set $A\cup \{b\}$ there are at most 2 hyperedges and one of them is $A$.

This is modified in our proof the following way. After picking $A$ and $b$, it is possible that there is a copy of $\cS^rK_3$ with vertex set $A\cup \{b\}$. In that case, the lower bound $r-1$ is replaced by $0$. This happens three times for every copy of $\cS^rK_3$, thus our lower bound is $|E(\cH)|(n-r)(r-1)-3t(r-1)$.

An upper bound in an $\cS^rK_3$-free hypergraph $\cH'$  is obtained in \cite{GS} the following way. We let $C_i$ denote the $i$th $(r-1)$-set in an arbitrary ordering, and $a_i$ denote the number of hyperedges containing $C_i$. Then we have $\sum_{i=1}^{\binom{n}{r-1}} a_i=r|E(\cH')|$. We count the pairs by picking their intersection first, thus there are $\sum_{i=1}^{\binom{n}{r-1}} a_i(n-r+1-a_i)$ pairs. This is equal to $(n-r+1)r|E(\cH')|-\sum_{i=1}^{\binom{n}{r-1}} a_i^2$. Then we obtain a lower bound $\frac{1}{\binom{n}{r-1}}(r|E(\cH')|)^2$ on $\sum_{i=1}^{\binom{n}{r-1}} a_i^2$ using Jensen's inequality. There is equality here if $a_i=n/r$ for every $i$. 

This is modified in our proof the following way. Consider a suspended triangle with suspending set $S$ and let $G$ denote its link graph. Then by our assumption, there is a vertex $v$ of degree less than $cn$ in $G$. Let $C_i=S\cup \{v\}$, then we have $a_i<cn$. Let us assign to each suspended triangle such an $(r-1)$-set which contains the suspending vertices of the suspended triangle and is contained in less than $cn$ hyperedges. 

We claim that for any $\varepsilon>0$, if $n$ is large enough then each $(r-1)$-set $C_i$ is assigned to at most $\varepsilon n^2$ suspended triangles. Indeed there are $r-1$ ways to pick the set $S$ of suspending vertices. After that, each suspended triangle that $C_i$ is assigned to has the same set of suspending vertices, thus creates a triangle in the link graph of $S$. That link graph is $F$-free, thus by our assumptions, it contains at most $\varepsilon n^2$ triangles.

We obtained that there are $q\ge t/\varepsilon n^2$ $(r-1)$-sets $C_i$ with $a_i<cn$, without loss of generality they are $C_1,\dots,C_q$. We use this property when bounding $\sum_{i=1}^{\binom{n}{r-1}} a_i^2$. Rather than applying Jensen's inequality directly, we will use that $\sum_{i=q+1}^{\binom{n}{r-1}} a_i^2$ is minimized if all the $a_i$'s with $i>q$ are equal.
We will also use the trivial bound $\sum_{i=1}^{q} a_i^2\ge 0$.

Observe first that since $\cH$ has at least $n\binom{n}{r-1}/r^2$ hyperedges, we have that the average $a_i$ is at least $n/r$ (and is $(1+o(1))n/r$ because of the asymptotic bound from Theorem \ref{asy}). Compared to the case each $a_i$ are equal, 

we have that $\sum_{i=1}^{q} a_i$ is smaller by at least $q(n/r-cn)$, thus $\sum_{i=q+1}^{\binom{n}{r-1}} a_i$ is larger by at least $q(n/r-cn)$. 
To minimize $\sum_{i=q+1}^{\binom{n}{r-1}} a_i^2$, we have to divide $q(n/r-cn)$ evenly, thus $a_i\ge n/r+q(n/r-cn)/(\binom{n}{r-1}-q)$ for $i>q$. Compared to the bound obtained if each $a_1=a_2=\dots =a_{\binom{n}{r-1}}$, we gain at least $2nq(n/r-cn)/r\ge c_2t/\varepsilon$ for some $c_2:=c_2(F,r)$.

As we have described, in \cite{GS} one obtains the inequality $|E(\cH)|(n-r)(r-1)\le (n-r+1)r|E(\cH)|-\frac{1}{\binom{n}{r-1}}(r|E(\cH')|)^2$. Here we obtain instead the inequality $|E(\cH)|(n-r)(r-1)\le (n-r+1)r|E(\cH)|-3t(r-1)-\frac{1}{\binom{n}{r-1}}(r|E(\cH')|)^2-c_2t/\varepsilon$. If $t>1$ and $\varepsilon$ is sufficiently small, then we subtract more on the right hand side, therefore the second inequality implies the first. Then the same way as in \cite{GS}, one obtains the desired bound on $|E(\cH)|$.  
\end{proof} 

Now we are ready to prove Theorem \ref{shar}. Recall that it states that $\ex_4(n,\S^4C_{2k+1})\le n\binom{n}{3}/16$ for $n$ sufficiently large.

\begin{proof}[\bf Proof of Theorem \ref{shar}]
    We will apply Theorem \ref{3chrom}. We have $\ex(n,K_3,C_{2k+1})=o(n^2)$ by a theorem of Gy\H ori and Li \cite{GyoriLi}. Observe that if $t$ is sufficiently large, then $Q(t)$ and $H(t)$ both contain $C_{2k+1}$. Therefore, by Proposition \ref{triv} there is a $c<1/4$ such that if an $n$-vertex $C_{2k+1}$-free graph $G$ contains a triangle, then it contains a vertex of degree at most $cn$. This implies that we can apply Theorem \ref{3chrom} and obtain the desired upper bound. 
    
    The sharpness of the bound for infinitely many $n$ follows from the construction of Gunderson and Semeraro \cite{GS} and Proposition 23 in \cite{GS}, as we have already mentioned.
\end{proof}

We remark that the above proof works for some other 3-chromatic graphs $F$ in place of odd cycles: those that are subgraphs of both $Q(t)$ and $H(t)$ for some $t$ and have $\ex(n,K_3,F)=o(n^2)$. These hold for example if $F$ contains exactly one cycle and that is a triangle, using a result of Alon and Shikhelman \cite{AS} which shows that $\ex(n,K_3,F)=O(n)$ in this case.

\section{Proofs of Proposition \ref{simpli} and Theorem \ref{fores}}

First, let us begin with the proof of Proposition \ref{simpli}, which states that $\ex_r(n,\cS^rF)\le \binom{n}{r-k}\ex_k(n-r+k,\cS^kF)/\binom{r}{k}$.

\begin{proof}[\bf Proof of Proposition \ref{simpli}]
        Let $\cH$ be an $n$-vertex $\cS^rF$-free $r$-graph. Obviously,
        for any $(r-k)$-set $A$ of vertices of $\cH$, the link hypergraph $\cL(\cH,A)$ is $\cS^kF$-free. Note that there are $\binom{n}{r-k}$ ways of picking $A$ and at most $\ex_k(n-r+k,\cS^kF)$ ways to extend $A$ to a hyperedge of $\cH$ for each $A$. Since each hyperedge of $\cH$ is counted $\binom{r}{r-k}$ ways, we have $|E(\cH)|\le \binom{n}{r-k}\ex_k(n-r+k,\cS^kF)/\binom{r}{k}$. This completes the proof.
\end{proof}

Let $\cH$ be an $r$-graph and $U\subseteq V(\cH)$. Denote by $H[U]$ the induced subhypergraph of $\cH$ on $U$. 
Let us continue with the proof of Theorem \ref{fores}. 
Recall that it states that $\ex_3(n,\S^3(P_3\cup K_2)\le \left\lfloor\binom{\lfloor (3n-1)/4\rfloor}{2}/3+(3n-1)(n+1)/32\right\rfloor$, and this bound is sharp if $n=8k+5$, and always sharp apart from an additive constant.

\begin{proof}[\bf Proof of Theorem \ref{fores}]
The lower bound is given by the following hypergraph. We take an $(m,3,2,1)$-design, also known as a Steiner triple system, and let $M$ denote its vertex set. Let $S$ denote the rest of the vertices and we partition $S$ into sets $S_1,\dots,S_{\lfloor(n-m)/2\rfloor}$ of size 2. Each 3-set containing a set $S_i$ and a vertex of $M$ is also a hyperedge. Observe that the link graph of a vertex of $M$ is a matching, and the link graph of a vertex of $S$ is a star, thus the resulting hypergraph is $\cS^3(P_3\cup K_2)$-free. The number of hyperedges is $\frac{m(m-1)}{6}+m\lfloor\frac{n-m}{2}\rfloor$. 

If $n=8k+5$, then we take $m=\lfloor\frac{3n-1}{4}\rfloor=6k+3$. Note that $n-m$ is even. By the Existence Theorem \cite{Ke2}, there exists an $(m,3,2,1)$-design. 
The number of hyperedges is
$$\frac{m(m-1)}{6}+\frac{m(n-m)}{2}=\binom{\lfloor\frac{3n-1}{4}\rfloor}{2}/3+(3n-1)(n+1)/32=12k^2+14k+4.$$

Now we consider the upper bound. Let $\cH$ be an $n$-vertex $\cS^3(P_3\cup K_2)$-free $3$-graph. Denote by $M$ the set of vertices with link graph being a matching (including $M_1$), and let $S=V(\cH)\backslash M$. It is easy to see that for any $v\in S$, the link graph $L_v$ contains a subgraph $P_3$. Let $|M|=m$. Then $|S|=n-m$. We call a hyperedge of $\cH$ \textit{$MMS$-type} if it contains two vertices from $M$ and one vertex from $S$. Similarly, we can define $MMM$-type, $MSS$-type and $SSS$-type hyperedges. 

\begin{clm}\label{clm1}
There is no $MMS$-type hyperedge in $\cH$. 
\end{clm}

\begin{proof}[\bf Proof of Claim]

Assume that there is such a hyperedge with $u,v\in M$, $w\in S$. Then $u,v\in V(\cL_w)$. Recall that for any $v\in S$, the link graph $\cL_v$ contains a subgraph $P_3$. Without loss of generality, let $P_3=v_1v_2v_3\subseteq \cL_w$. If $\{u,v\}\cap\{v_1,v_2,v_3\}=\emptyset$, then there is a copy of $\cS^3(P_3\cup K_2)$, a contradiction. If $|\{u,v\}\cap\{v_1,v_2,v_3\}|=1$ or $|\{u,v\}\cap\{v_1,v_2,v_3\}|=2$, then we have $d_{\cL_w}(u)\geq2$ or $d_{\cL_w}(v)\geq2$. This implies that $d_{\cL_u}(w)\geq2$ or $d_{\cL_v}(w)\geq2$. Hence the link graph of $u$ or $v$ contains two adjacent edges, a contradiction to $u,v\in M$. 
\end{proof}

We claim that $\cH[M]$ is a linear $3$-graph, i.e. any two vertices are contained in at most one hyperedge. Otherwise, the link graph contains two adjacent edges, a contradiction. Thus, the number of $MMM$-type hyperedges is at most $\binom{m}{2}/\binom{3}{2}=m(m-1)/6$. 

\begin{clm}\label{clm2}
For any $u\in M, v\in S$, there is at most one hyperedge that contains $u$ and $v$. 
\end{clm}

\begin{proof}[\bf Proof of Claim]
Assume there are two hyperedges, $e_1$ and $e_2$, containing $u$ and $v$, denoted as $uvw_1,uvw_2$, respectively. By Claim \ref{clm1}, we have $w_1,w_2\in S$. Then $d_{\cL_u}(v)\geq2$. Hence the link graph of $u$ contains two adjacent edges, a contradiction. 
\end{proof}

Let $S_1\subseteq S$ denote the set of vertices with link graph being a star with at least 2 edges, and let $S_2=S\setminus S_1$. 

\begin{clm}\label{clm3}
For any $u\in S_2$, $3\leq|V(\cL_u)|\leq4$. 
\end{clm}

\begin{proof}[\bf Proof of Claim]
Obviously, $3\leq|V(\cL_u)|$ for any $u\in S_2$. Suppose there is a vertex $u\in S_2$ such that $|V(\cL_u)|\geq5$. Since the link graph has no isolated vertices and $P_3\subseteq \cL_u$, it is easy to see that $P_3\cup K_2\subseteq \cL_u$, a contradiction.
\end{proof}

\begin{clm}\label{clm4}
$|E(\cH[S_1])|=0$.
\end{clm}

\begin{proof}[\bf Proof of Claim]
Assume that $uvw$ is a hyperedge inside $S_1$. Then for each vertex $u$ of the hyperedge, another vertex $v$ is the center of the link graph, thus each hyperedge containing $u$ also contains $v$. Let us take a directed edge $uv$ then. This way we take three directed edges inside $\{u,v,w\}$, an edge going out from each vertex. Then, there is a two-edge directed path here, say from $u$ to $v$ and then from $v$ to $w$. Then each hyperedge containing $u$ also contains $v$ and $w$, thus there is only one such hyperedge, a contradiction. 
\end{proof}

Next, we consider the number of $MSS$-type and $SSS$-type hyperedges. We divide $MSS$-type and $SSS$-type hyperedges into two parts: those containing vertices in $S_2$ are denoted as $E_2$, and the rest as $E_1$. Therefore, $|E_2|\leq\sum_{v\in S_2}d_\cH(v)$. 
By Claim \ref{clm3}, $|E(\cL_u)|\leq |E(K_4)|=6$ for any $u\in S_2$. Then $d_\cH(u)\leq 6$ for any $u\in S_2$. Thus, the number of hyperedges in $E_2$ is at most $$\sum_{v\in S_2}d_\cH(v)\leq 6|S_2|.$$

Now we consider the number of hyperedges in $E_1$. By the definition of $E_1$ and Claim \ref{clm4}, the hyperedges in $E_1$ are all of the $MSS$-type. Therefore, we have $|E_1|\leq \frac{m|S_1|}{2}$ by Claim \ref{clm2}. 
Let $|S_2|=x$. Then $|S_1|=n-m-x$. Denote by $E_0$ the set of $MMM$-type hyperedges. Then 
\begin{eqnarray*}
|E(\cH)|&=& |E_0|+|E_1|+|E_2| \\
&\leq& \frac{m(m-1)}{6}+\frac{m(n-m-x)}{2}+6x  \\
&=& -\frac{1}{3}m^2+(\frac{n}{2}-\frac{x}{2}-\frac{1}{6})m+6x. 
\end{eqnarray*} 
The above expression reaches its maximum value when $m=\frac{3(n-x)-1}{4}$. Therefore, 
\begin{eqnarray*}
|E(\cH)|\leq \frac{3}{16}x^2+(\frac{49}{8}-\frac{3n}{8})x+(\frac{3}{16}n^2-\frac{n}{8}+\frac{1}{48}). 
\end{eqnarray*} 

When $n>33$, the above expression reaches its maximum value when $x=0$. Thus,  
\begin{eqnarray*}
|E(\cH)|&\leq& \frac{3}{16}x^2+(\frac{49}{8}-\frac{3n}{8})x+(\frac{3}{16}n^2-\frac{n}{8}+\frac{1}{48}) \\
&\leq& \frac{3}{16}n^2-\frac{n}{8}+\frac{1}{48} \\
&=& \frac{1}{3}\binom{\frac{3n-1}{4}}{2}+\frac{(3n-1)(n+1)}{32}. 
\end{eqnarray*}

Furthermore, if $n=8k+5$, then we have $$\frac{1}{3}\binom{\frac{3n-1}{4}}{2}+\frac{(3n-1)(n+1)}{32}=12k^2+14k+4+1/12.$$
Since $k$ is an integer, we have 
$$|E(\cH)|\leq 12k^2+14k+4=\binom{\lfloor\frac{3n-1}{4}\rfloor}{2}/3+(3n-1)(n+1)/32.$$

This completes the proof. 
\end{proof}

Let us remark that the 4-uniform case of this problem also seems exciting. An $(n,4,3,1)$-design is a simple $\cS^4(P_3\cup K_2)$-free construction, with $(1+o(1))n^3/24$ hyperedges. We have another construction of roughly the same size. Let us take sets $M$ and $S$ of size $n/2$. We partition the triples inside $M$ to $n/2$ \textit{independent} subhypergraphs, i.e., subhypergraphs in which any two hyperedges  share at most one vertex. It is well-known that we can do this, it is a coloring of the Johnson graph, see e.g., \cite{eb}. For these subhypergraphs we pick distinct vertices from $S$ and take the 4-edges formed by a vertex from $S$ and a hyperedge of the corresponding subhypergraph. Moreover, we take an $(n/2,3,2,1)$-design on $S$ and we take all the 4-edges formed by a hyperedge of this design and a vertex of $M$. It is not hard to see that the resulting hypergraph is $\cS^4(P_3\cup K_2)$-free and has $(1+o(1))n^3/24$ hyperedges.

Assume now that we can partition the triples of an $m$-element set to (roughly) $m/2$ independent subhypergraph, and (roughly) $\binom{m}{3}/8$ 4-cliques. We do not know whether such a partition exists. If it does, we take a set $M$ of $m=2n/3$ vertices and such a partition on $M$. We also take a set $S$ of $n/3$ vertices, and assign each such vertex to one of the independent subhypergraphs in $M$. We take the 4-edges formed by a vertex from $S$ and a hyperedge of the corresponding subhypergraph. We also take the 4-edges formed by the 4-cliques in the partition inside $M$. Finally, we take an $(n/3,3,2,1)$-design on $S$ and we take all the 4-edges formed by a hyperedge of this design and a vertex of $M$. It is not hard to see that the resulting hypergraph is $\cS^4(P_3\cup K_2)$-free and has $(1+o(1))(\binom{m}{3}/2+\binom{m}{3}/8+m\binom{n-m}{2}/3)=(1+o(1))7n^3/162$ hyperedges. The upper bound that follows from Proposition \ref{simpli} and Theorem \ref{fores} is $(1+o(1))3n^3/64$.

\ \

\bigskip
\textbf{Funding}: 
The research of Cheng is supported by the National Natural Science Foundation of China (Nos. 12131013 and 12471334), Shaanxi Fundamental Science Research Project for Mathematics and Physics (No. 22JSZ009).

The research of Gerbner is supported by the National Research, Development and Innovation Office - NKFIH under the grant KKP-133819.

The research of Zhou is supported by the National Natural Science Foundation of China (Nos. 11871040, 12271337, 12371347) and the China Scholarship Council (No. 202406890088).

\end{document}